\documentclass[a4paper,11pt]{amsart}
\usepackage[foot]{amsaddr}

\usepackage[a4paper,top=3cm,bottom=3cm,
           left=3cm,right=3cm,marginparwidth=60pt]{geometry}
\usepackage{comment}
\usepackage{xcolor}
\usepackage{amsmath}
\usepackage{mathtools}
\usepackage[all]{xy}
\usepackage[utf8]{inputenc}
\usepackage{varioref}
\usepackage[normalem]{ulem}
\usepackage{amsfonts}
\usepackage{amssymb}
\usepackage{bbm}
\usepackage{esint}
\usepackage{graphicx}
\usepackage{tikz}
\usepackage{empheq}
\usepackage{enumitem}
\usepackage{tikz-cd}
\usetikzlibrary{matrix,arrows,decorations.pathmorphing}
\usepackage{mathrsfs}
\usepackage[hypertexnames=false,backref=page,pdftex,
 	pdfpagemode=UseNone,
 	breaklinks=true,
 	extension=pdf,
 	colorlinks=true,
 	linkcolor=blue,
 	citecolor=red,
 	urlcolor=blue,
 ]{hyperref}

\usepackage{soul}

\def\tiago#1{{\color{blue}\textbf{Tiago: }#1}}
\definecolor{brickred}{rgb}{0.8, 0.25, 0.33}

\usepackage[textsize=small]{todonotes}
\newcommand\Luca[1]{\todo[color=yellow!40]{#1}}
\newcommand\Lucaline[1]{\todo[inline,color=yellow!40]{#1}}

\newcommand{\bbP}{{\mathbb P}}
\newcommand{\bbR}{{\mathbb R}}

\newcommand{\Bl}{\operatorname{Bl}}

\theoremstyle{plain}

\newtheorem{theorem}[subsection]{Theorem}
\newtheorem{definition}[subsection]{Definition}

\newtheorem{lemma}[subsection]{Lemma}
\newtheorem{corollary}[subsection]{Corollary}

\newtheorem{proposition}[subsection]{Proposition}

\theoremstyle{remark}
\newtheorem{example}[subsection]{Example}
\newtheorem{remark}[subsection]{Remark}

\title[K-stability]{On K-stability of $\mathbb P^3$ blown up along a smooth genus $2$ curve of degree $5$}

\author{Tiago Duarte Guerreiro}
\address[Tiago Duarte Guerreiro]{Laboratoire de Mathématiques d’Orsay, Université Paris-Saclay,Michel Magat, Bat. 307, 91405 Orsay, France}
\email{tiago.duarte-guerreiro@universite-paris-saclay.fr}

\author{Luca Giovenzana}
\address[Luca Giovenzana]{Department of Pure Mathematics\\ University of Sheffield\\ Hicks Building, Hounsfield Road\\ Sheffield, S3 7RH\\ UK}
\email{l.giovenzana@sheffield.ac.uk}

\author{Nivedita Viswanathan}
\address[Nivedita Viswanathan]{School of Physical and Chemical Sciences, Queen Mary University of London, Mile End, London E1 4NS\\ UK }
\email{N.Viswanathan@qmul.ac.uk}

\makeatletter
\@namedef{subjclassname@2020}{%
 \textup{2020} Mathematics Subject Classification}
\makeatother



\usepackage{framed}

\begin{document}

\begin{abstract}	
We prove K-stability for infinitely many smooth members of the family 2.19 of the Mukai-Mori classification.
\end{abstract}

\maketitle

\section{Introduction}\label{section introduction}

The notion of K-stability was introduced in \cite{tian97} as a criterion to detect the existence of K\"ahler-Einstein metrics on complex Fano manifolds, which is of fundamental importance in complex geometry. In the seminal series of papers \cite{CDS1,CDS2,CDS3, Tian}, the authors prove that such an existence has an algebro-geometric characterisation, known as K-polystability, and hence solving the famous Yau-Tian-Donaldson conjecture. 

From then on, a great deal of work has been carried out to verify K-stability of Fano manifolds. Of particular importance is the work of Abban and Zhuang \cite{AZ21} where the authors introduce a new powerful inductive framework. These new techniques have been most notably used in \cite{ACCFKMGSSV} where the authors verified K-(poly)stability of general members of the 105 families of smooth Fano threefolds and much work has been put into extending this to all members of many families (see \cite{LiuXu19, Kentorank3deg28, denisova2022kstability,cheltsov2022kstable,belousov2022kstability, CheltsovPark2022, Liu-Rank2Deg14,cheltsov2023kstable, CheltsovFujitaKishimotoPark,Malbon, 215}). 

Despite the extensive efforts, K-polystability of every member of few families still remain to be proven.  This paper arises in our attempt to completely answer the question on K-polystability of every member in families of Fano 3-folds that are given by blowing up of curves on smooth quadrics. 

In particular, consider the following families of Fano 3-folds.
Let $\mathcal{C}\subset \mathbb P^3$ be a curve lying in a quadric $Q \simeq \mathbb P^1 \times \mathbb P^1$ of bidegree $(a,b)$ with $a\leq b$ and $\alpha \colon X \rightarrow \mathbb P^3$ be the blowup along $\mathcal{C}$. Then $X$ is a smooth Fano threefold if and only if $b\leq 3$ by \cite[Proposition~3.1]{LamyBlancWeak}, giving six non-isomorphic deformation families, that is, 2.15, 2.19, 2.22, 2.25, 2.30, 2.33 from the Mori-Mukai classification. The members of the families 2.30 and 2.33 are proven to be K-unstable in \cite{ACCFKMGSSV}. While \cite{215} proves K-stability of every smooth member of family 2.15, and \cite{ACCFKMGSSV} and \cite{CheltsovPark2022} proves results on K-polystability of members of family 2.25 and family 2.22, respectively.  

This paper deals with smooth members of the remaining family 2.19. Recall that every smooth Fano threefold of the family 2.19 can be obtained as the blow-up of $\mathbb P^3$ along a smooth curve $\mathcal{C}$ of genus 2 and degree 5. Then $\mathcal{C}$ is contained in a unique quadric $Q$. Let 
\begin{itemize}
    \item $X=\mathrm{Bl}_{\mathcal{C}}\mathbb{P}^3$ be the associated Fano threefold.
    \item $E$ the exceptional divisor of the blow-up morphism $\alpha\colon X\to \mathbb P^3$.
    \item $\widetilde Q$ be the strict transform of $Q$.
\end{itemize}

 Using this and the following from \cite{kentoodaka,blumjonsson}
\begin{center}
    $X$ is K-stable if and only if $\delta_p(X)>1$ for every $p\in X$,
\end{center} we attempt to establish the K-stability of a given member $X$ of family 2.19  by estimating $\delta_p$ for every point $p\in X$ in Section \ref{section:219}.

Unfortunately, we cannot give an estimate of $\delta_p$ for every point in $X$, but our main result (Theorem \ref{thm:main}) is enough to show K-stability of infinitely many new examples of smooth Fano threefolds in the family 2.19. 


\begin{theorem} \label{thm:main}
Let $X$ be a smooth Fano threefold given by the blow-up of a genus 2 curve $\mathcal C\subset \mathbb P^3$ lying on a smooth quadric. Then
\begin{center}
    $\delta_p(X)>1$ for all points $p\not\in E\setminus\widetilde Q$.
\end{center}
\end{theorem} 

As an application of the above result, we can prove the K-stability of many members of family 2.19 that also satisfy an additional condition. 
\begin{corollary} \label{cor:nofixedpts-noproof}
    Let $X$ be the blow-up of $\mathbb P^3$ in a smooth genus 2 curve $\mathcal C$ lying on a smooth quadric. If the group Aut$(\mathbb P^3, \mathcal C)$ has no fixed points on $\mathcal C$, then $X$ is K-stable. 
\end{corollary}
Examples of families of curves that satisfy this is given in Example \ref{ex:autnofixedpoint}. Inspired from \cite{abban2024kstabilitypointlessdelpezzo}, we also show 
\begin{corollary}\label{cor:pointless-noproof}
    Let $X$ be the blow-up of $\mathbb P^3$ along a smooth curve $\mathcal C$ whose equation in $Q\cong\mathbb P^1\times \mathbb P^1$ is given by 
    \begin{align*}
    s_0^2f_3(t_0,t_1)+s_1^2g_3(t_0,t_1) \in \mathbb Q[s_0,s_1,t_0,t_1]
    \end{align*}
    where $f_3$ and $g_3$ are homogeneous polynomials of degree 3.  If $f_3$ and $g_3$ have no rational solutions, then $X$ is K-stable.
\end{corollary}
It is then immediate to find \textit{infinitely} many examples of polynomials satisfying the assumptions of Corollary~\ref{cor:pointless-noproof}.

\begin{example}
     Let 
    $$
    \mathcal{C} \colon (s_0^2(t_0^3-\eta_0t_1^3)+s_1^2(t_0^3-\eta_1t_1^3)=0) \subset \mathbb P^1_{\mathbb Q}\times \mathbb P^1_{\mathbb Q}
    $$
    where $\eta_0 \not = \eta_1$ are integers which are not perfect cubes. It follows from Corollary \ref{cor:pointless} that the blowup of $X$ along $\mathcal{C}$ is K-stable.
\end{example}

The paper is split as follows. Section~\ref{Infinitely many K-stable members} contains the proofs of Corollary~\ref{cor:nofixedpts-noproof} and Corollary~\ref{cor:pointless-noproof} together with infinitely many new examples of K-stable 3-folds. The theory of Abban-Zhuang from \cite{AZ21} that is needed to prove Theorem \ref{thm:main} is detailed in Section \ref{section: AZ Theory}. The main Theorem \ref{thm:main} is proven in Section \ref{section:219} where the birational geometry of members of family 2.19 as well as the notation we use for the computations henceforth in the paper, is explained.\medskip

While completing this work, we were informed that all K-stable members are announced to be K-stable in a forthcoming paper \cite{Zhao24} using different methods.
\medskip

\textit{Acknowledgments:} The authors would like to extend their heartfelt gratitude to Prof. Ivan Cheltsov for suggesting this problem and for providing constant guidance when in need. Sincere thanks to Kento Fujita, Anne-Sophie Kaloghiros, Tony Shaska, Cesar Hilario for being generous with their time and expertise. The authors are grateful to Junyan Zhao for having informed them about his work on the same problem, approached from a different perspective. The second author is grateful to Franco Giovenzana for discussions and to Ciro Ciliberto for Remark~\ref{rmk:Ciro}. 

The first author was partially supported by Engineering and Physical Sciences
Research Council (EPSRC)  EP/V048619/1, EP/V055399/1 as well as ERC StG Saphidir No. 101076412 and PEPS Jeunes chercheurs et jeunes chercheuses, 2024. The second author is a member of INdAM GNSAGA. The third author was supported by the Engineering and Physical Sciences
Research Council (EPSRC) New Horizons Grant EP/V048619/1 as well as Queen Mary University of London.

\section{Infinitely 
 many examples of K-stable 3-folds in 2.19}\label{Infinitely many K-stable members}
In this section, we prove Corollary \ref{cor:nofixedpts-noproof} and Corollary \ref{cor:pointless-noproof}, which is an integral component to providing infinite examples of members of 2.19 that are K-stable.

\begin{corollary} (also Corollary \ref{cor:nofixedpts-noproof}) \label{cor:nofixedpts}
    Let $X$ be a smooth Fano threefold given by the blow-up of a genus 2 curve $\mathcal C\subset \mathbb P^3$ lying on a smooth quadric. If the group Aut$(\mathbb P^3, \mathcal C)$ has no fixed points in $\mathcal C$, then $X$ is K-stable. 
\end{corollary}
\begin{proof}
By \cite[Corollary~4.14]{equivkstab} in order to prove the K-polystability of $X$, it is sufficient to prove that $A_X(F)>S_X(F)$, see Section~\ref{section: AZ Theory} for notation, for every $G$-equivariant prime divisor $F$ over $X$, for any $G$ a reductive subgroup of Aut$(X)$. If the centre $Z$ of $F$ is not contained in $E\setminus \widetilde Q$, then there exists a point $p$ in $Z$ for which Theorem~\ref{thm:main} applies giving the bound $\delta_Z(X)\geq\delta_p(X)>1$. Since by assumption there are no fixed points on the curve $\mathcal{C}$ for the action of Aut$(\mathbb P^3, \mathcal C)$, there are no $G$-equivariant prime divisors with centre a point in $E\setminus\widetilde Q$. Thus it only remains to check prime divisors $F$ with centre a curve contained in $E\setminus \widetilde Q$ which is checked in Lemma~\ref{lem:centreCurve}. This proves that $X$ is K-polystable. Since $\mathrm{Aut}(X)$ is finite by \cite[Corollary~12.1]{3foldsInfAut}, it follows that $X$ is K-stable (see \cite[Corollary~1.3]{blum2019uniqueness}), thus proving the claim. 
\end{proof}

\begin{example}\label{ex:autnofixedpoint}
    Let 
    $$
    \mathcal{C} \colon (s_0^2(at_0^3+bt_0^2t_1+ct_0t_1^2+dt_1^3)+s_1^2(dt_0^3+ct_0^2t_1+bt_0t_1^2+at_1^3)=0) \subset \mathbb P^1\times \mathbb P^1
    $$
    where $a, \,b,\, c,\, d \in \mathbb C$ are such that $\mathcal{C}$ is smooth.  Consider the involution $$
    \tau \colon ([s_0:s_1],[t_0:t_1]) \mapsto ([s_1:s_0],[t_1:t_0]).
    $$
    Then $\mathcal{C}$ is $\tau$-invariant. Moreover, the fixed locus of $\tau$ consists of the four points 
    \begin{align*}
&p_{11} : ([1:1],[1:1]), &&p_{12} : ([1:-1],[1:-1])         \\
&p_{21} : ([1:-1],[1:1]), &&p_{22} : ([1:1],[1:-1])
    \end{align*}
 However notice that the hyperelliptic involution $\sigma \colon \mathbb P^1 \times \mathbb P^1 \rightarrow \mathbb P^1 \times \mathbb P^1$ of $\mathcal{C}$ 
  given by
$$
 ([s_0:s_1] , [t_0:t_1]) \mapsto ([s_0:-s_1] , [t_0:t_1]).
$$
 swaps these points, namely, $\sigma(p_{11}) = p_{21}$ and    $\sigma(p_{12}) = p_{22}$. Let $G = \langle \sigma, \tau \rangle \simeq \mu_2 \times \mu_2 \subset \mathrm{Aut}(\mathbb P^3,\mathcal{C})$. Let $X$ be the blowup of $\mathbb P^3$ along $\mathcal{C}$. By Corollary \ref{cor:nofixedpts} it follows that $X$ is K-stable.
\end{example}
\begin{corollary} (also Corollary \ref{cor:pointless-noproof}) \label{cor:pointless}
Let $X$ be the blow-up of $\mathbb P^3$ along a smooth curve $\mathcal C$ whose equation in $Q\cong\mathbb P^1\times \mathbb P^1$ is given by 
    \begin{align*}
    s_0^2f_3(t_0,t_1)+s_1^2g_3(t_0,t_1) \in \mathbb Q[s_0,s_1,t_0,t_1]
    \end{align*}
    where $f_3$ and $g_3$ are homogeneous polynomials of degree 3.  If $f_3$ and $g_3$ have no rational solutions, then $X$ is K-stable.
\end{corollary}

\begin{proof}
    Since $X$ is defined over $\mathbb Q$, by \cite[Corollary~4.14]{equivkstab} in order to prove K-polystability of $X$  it is enough to prove $A_X(F)>S_X(F)$, see Section~\ref{section: AZ Theory} for notation, for every $G$-invariant geometrically irreducible divisor over $X$, where $G<\mathrm{Aut}(X)$. Let $F$ be any such divisor. If its centre is not contained in $E\setminus \widetilde Q$, then $A_X(F)>S_X(F)$ by Theorem~\ref{thm:main}. While if $F$ has centre on a curve in $E\setminus \widetilde Q$ one concludes by Lemma~\ref{lem:centreCurve}. In order to conclude we show that there are no geometrically irreducible $G$-invariant divisors with centre a point in $E\setminus \widetilde Q$. Indeed, if there was such a divisor, its centre $p\in E$ maps to a geometrically irreducible point in $\mathcal C\subset \mathbb P^3$ fixed by any subgroup $G<\mathrm{Aut}(\mathcal C, \mathbb P^3)$, but there is no such a point when $\mathcal C$ is as in the assumptions:
    
    For this, we consider the embedding $\mathbb P^1\times\mathbb P^1\to \mathbb P^3$ given by
\begin{align*}
([s_0:s_1] , [t_0:t_1]) \mapsto [s_0t_0:s_0t_1:s_1t_0:s_1t_1],
\end{align*}
which identifies $\mathbb P^1\times \mathbb P^1$ with the quadric $Q=(x_0x_3 - x_1x_3=0)$. The hyperelliptic involution $\sigma$ of $\mathcal C$ given by
\begin{align*}
     ([s_0:s_1] , [t_0:t_1]) \mapsto ([s_0:-s_1] , [t_0:t_1])
\end{align*}
is the restriction of the projectivity of $\mathbb P^3$ given by
\begin{align*}
    [x_0:x_1:x_2:x_3] \mapsto [x_0:x_1:-x_2:-x_3].
\end{align*}
Its action on $\mathcal C$ acts trivially on the set
\begin{align*}
    \{f_3(t_0,t_1) = 0 \} \cup \{g_3(t_0,t_1) = 0 \},
\end{align*}
but by assumption they do not define any geometrically irreducible point in $\mathcal C$.
\end{proof}

\section{Preliminaries}\label{section: AZ Theory}

\subsection{Definition of K-stability}
In this section we recall the definition of K-stability and the main results used in order to prove Theorem~\ref{thm:main}.

\begin{definition}
Let $\Delta$ be an effective $\mathbb{Q}$-divisor on a normal projective variety $X$ for which $K_X+\Delta$ is $\mathbb{Q}$-Cartier. We say that $(X,\Delta)$ is a \textbf{log Fano pair} if $(X, \Delta)$ is klt and $-(K_X + \Delta)$ is ample. If $\Delta = 0$, we call $(X, 0)$ a Fano variety and denote it by $X$.
 \end{definition}

We recall the notion of stability threshold (or $\delta$-invariant) introduced in \cite{kentoodaka}.
 \begin{definition}
 Let $(X,\Delta)$ be a log Fano pair, and let $f : Y \rightarrow X $ be a projective birational morphism such that $Y$ is normal and let $E$ be a prime divisor on $Y$. Let $L$ be an ample $\mathbb Q$-Cartier divisor on $X$. We set,
\[
A_{X,\Delta}(E)=1+\mathrm{ord}_E\big(K_Y-f^*(K_X+\Delta)\big), \quad S_L(E)=\frac{1}{L^n}\int_0^{\infty} \mathrm{vol}(f^*(L)-uE)du.
\]
We define the \textbf{stability threshold} as
\[
\delta(X,\Delta;L)=\inf_{E/X}\frac{A_{X,\Delta}(E)}{S_L(E)}
\]
where the infimum runs over all prime divisors over $X$. For an irreducible subvariety $Z \subset X$, we define the \textbf{local stability threshold} as
\[
\delta_Z(X,\Delta;L)=\inf_{\substack{E/X \\ Z \subset C_X(E)}}\frac{A_{X,\Delta}(E)}{S_L(E)}
\]
where the infimum runs over all prime divisors over $X$ whose centres on $X$ contain $Z$.
 \end{definition}
 
It is proved in \cite{kentoodaka,blumjonsson} that the following equivalence holds,
\[
\delta(X,\Delta)>1 \iff (X,\Delta)\,\, \text{is K-stable}.
\]
We will, in fact, take this to be our definition of K-stability of a Fano variety.  In particular, in this paper we use
\[
\delta(X,\Delta;L)= \inf_{p \in X}\delta_p(X,\Delta;L),
\]
for a point $p \in X$ and take $\Delta=0,\ L=-K_X$. 
\begin{definition}[{\cite[Definition~1.1]{kentoplt}}]
    Let $\Delta$ be an effective $\mathbb Q$-divisor on $X$ and $(X, \Delta)$ be a klt pair. A prime divisor $Y$ over $X$ is said to be of \textbf{plt-type} over $(X, \Delta)$ if there is a projective birational morphism $\mu : \widetilde X \rightarrow X$ with $Y \subset \widetilde X$ such that $-Y$ is a $\mu$-ample $\mathbb Q$-Cartier divisor on $\widetilde X$ for which $(\widetilde X,\widetilde \Delta+ Y)$ is a plt pair where 
 the $\mathbb Q$-divisor $\widetilde \Delta$ is defined by\[
    K_{\widetilde X}+ \widetilde \Delta + (1-A_{X,\Delta}(Y))Y=\mu^*(K_X+\Delta). 
    \]
 \end{definition}

 \begin{remark}
     The morphism $\mu$ is completely determined by $Y$ and it is called the \textbf{plt-blowup} associated to $Y$.
 \end{remark}

 \subsection{Abban-Zhuang Theory}
In the following, we describe how we study the K-stability of Fano 3-folds $X$ in this paper. We do this by employing the Abban-Zhuang theory developed in \cite{AZ21} to estimate the local stability threshold $\delta_p$ for every point in $X$. We recall the main results we need by referring to the book \cite{ACCFKMGSSV}.

Given a smooth Fano threefold $X$, so that, in particular Nef($X$)=Mov($X$), and a point $p\in X$ we consider flags $p\in Z\subset Y \subset X$ where:
\begin{itemize}
\item $Y$ is an irreducible surface with at most Du Val singularities;

\item $Z$ is a non-singular curve such that $(Y,Z)$ is plt.
\end{itemize}
We denote by $\Delta_Z$ the different of the log pair $(Y,Z)$.

For $u\in \bbR$, we consider the divisor class $-K_X-uY$ and we denote by $\tau=\tau (u)$ its pseudoeffective threshold, i.e. the largest number for which $-K_X-uY$ is pseudoffective. For $u\in[0,\tau]$, let $P(u)$ (respectively $N(u)$) be the positive (respectively negative) part of its Zariski decomposition.
Since $Y\not\subset\mathrm{Supp} (N(u))$ we can consider the restriction $N(u)\vert_Y$ and define $N'_Y(u)$ to be its part not supported on $Z$, i.e. $N'_Y(u)$ is the effective $\mathbb R$-divisor such that $Z\not\subset \mathrm{Supp}(N'_Y(u))$ defined by:
\[
N(u)\vert_Y= d(u)Z + N'_Y(u)
\]
where $d(u):=$ord$_Z(N(u)\vert_Y)$.

We consider then  for every $u\in [0,\tau]$ the restriction $P(u)\vert_Y$ and denote by $t(u)$ the pseudoeffective threshold of the divisor $P(u)\vert_Y-vZ$, by $P(u,v)$ and $N(u,v)$ the positive and negative part of its Zariski decomposition.
Let $V^Y_{\bullet,\bullet}$ and $W_{\bullet,\bullet,\bullet}^{Y,Z}$ be the multigraded linear series defined in \cite[Page~57]{ACCFKMGSSV}.

Finally we can state the main tool we use to estimate the local $\delta$-invariant:
\begin{theorem}\cite[Theorem~1.112]{ACCFKMGSSV}\label{delta estimate}
\begin{align*}
\delta_p(X)\geq \mathrm{min} \left\{ \frac{1-\mathrm{ord}_p\Delta_Z}{S(W_{\bullet,\bullet,\bullet}^{Y,Z};p)},\  \frac{1}{S(V^Y_{\bullet,\bullet};Z)},\ \frac{1}{S_X(Y)}\right\}
\end{align*} where 
\begin{align}\label{surface-to-curve}
S(V^Y_{\bullet,\bullet};Z)=\frac{3}{(-K_X)^3}\int_0^{\tau}(P(u)^2 \cdot Y) \cdot \mathrm{ord}_Z(N(u)\vert_Y)du+\frac{3}{(-K_X)^3}\int_0^{\tau}\int_0^{\infty}\mathrm{vol}(P(u)\vert_Y-vZ)dvdu,
\end{align}
and 
\begin{align}\label{curve-to-point}
S(W_{\bullet,\bullet,\bullet}^{Y,Z};p)=\frac{3}{(-K_X)^3}\int_0^{\tau}\int_0^{t(u)}(P(u,v) \cdot Z)^2dvdu+F_p(W_{\bullet,\bullet,\bullet}^{Y,Z}),
\end{align}
 with 
\begin{align}\label{Fp}
F_p(W_{\bullet,\bullet,\bullet}^{Y,Z})=\frac{6}{(-K_X)^3}\int_0^{\tau}\int_0^{t(u)}(P(u,v) \cdot Z)\cdot \mathrm{ord}_p(N_Y'(u)\vert_Z+N(u,v)\vert_Z)dvdu.
\end{align}
\end{theorem}

The theorem above admits a slight generalization which allows to consider not only flags of varieties in $X$, but also over $X$. In particular,
let $X$ and $Y$ be as above, in order to estimate $\delta_p$ for $p\in Y$ it turns out to be useful to consider curves over $Y$. For this, let $\sigma: \widetilde Y\to Y$ be a plt blow-up of $Y$ in $p$ and denote by $\widetilde Z$ its exceptional divisor.
We consider the linear system $\sigma^*(P(u)\vert_Y)-v \widetilde Z$ and denote by $\widetilde t(u)$ its pseudoeffective threshold, i.e.
\begin{align*}
\widetilde t(u)=\mathrm{max}\{v\in \mathbb{R}_{\geq 0}\ : \ \sigma^{*}(P(u)\vert_{Y})-v\widetilde Z\ \mathrm{is\ pseudoeffective}\}.
\end{align*}

For every $v\in [0,\widetilde t(u)]$ we denote by $\widetilde P(u,v)$ and $\widetilde N(u,v)$ the positive and negative part of its Zariski decomposition. We also denote by $N'_{\widetilde Y}(u)$ the strict transform of the divisor $N(u)\vert_Y$. 

\begin{theorem}\label{blow-up-of-surface-formula}\cite[Remark~1.113]{ACCFKMGSSV}
\begin{align*}
\delta_p(X)\geq \mathrm{min} \left\{ \mathrm{min}_{q\in\widetilde Z}\frac{1-\mathrm{ord}_q\Delta_{\widetilde{Z}}}{S(W_{\bullet,\bullet,\bullet}^{Y,\widetilde{Z}};q)},\ 
\frac{A_Y(\widetilde{Z})}{S(V^Y_{\bullet,\bullet};\widetilde{Z})},\
\frac{1}{S_X(Y)}\right\}
\end{align*}
where:
\begin{equation}
\begin{aligned}\label{Rmk1.7.32-surface-to-curve}
S(V^Y_{\bullet,\bullet};\widetilde Z)=&\frac{3}{(-K_X)^3}\int_0^{\tau}(P(u)^2 \cdot Y) \cdot \mathrm{ord}_{\widetilde{Z}}(\sigma^*( N(u)\vert_Y))du\ +\\
&\frac{3}{(-K_X)^3}\int_0^{\tau}\int_0^{\infty}\mathrm{vol}\left(\sigma^*\left(P(u)\vert_Y\right)-v\widetilde Z\right)dvdu,
\end{aligned}
\end{equation}
and 
\begin{align}\label{Rmk1.7.32-curve-to-point}
S(W_{\bullet,\bullet,\bullet}^{Y,\widetilde{Z}};q)=\frac{3}{(-K_X)^3}\int_0^{\tau}\int_0^{\widetilde{t}(u)}(\widetilde{P}(u,v) \cdot \widetilde{Z})^2dvdu
&+F_q(W_{\bullet,\bullet,\bullet}^{Y,\widetilde{Z}}),
\end{align}
with 
\begin{align}\label{Rmk1.7.32-FP}
F_q(W_{\bullet,\bullet,\bullet}^{Y,\widetilde Z})=\frac{6}{(-K_X)^3}\int_0^{\tau}\int_0^{\widetilde t(u)}(\widetilde P(u,v) \cdot \widetilde Z)\cdot \mathrm{ord}_q(N_{\widetilde{Y}}'(u)\vert_{\widetilde Z}+\widetilde N(u,v)\vert_{\widetilde Z})dvdu.
\end{align}
\end{theorem}\bigskip\bigskip

\section{Estimates of $\delta$-invariants} \label{section:219}

In this section, we deal with smooth Fano manifolds $X$ of the family 2.19 in the Iskovskikh-Mori-Mukai's list.  

Recall that any such smooth member is a threefold of Picard number 2 isomorphic to the blow-up of $\bbP^3$ in a degree 5 curve $\mathcal C$ lying on a quadric $Q$. If the quadric $Q$ is smooth the curve has class $2\ell_1+3\ell_2$, where $\ell_1$ and $\ell_2$ are the two rulings of $Q\sim \mathbb P^1\times \mathbb P^1$. We review its birational geometry.

For this, let $\mathcal C\subset\bbP^3$ be such a curve on the quadric $Q = (f_2=0)$. We are interested in the study of the blow-up $X:=\Bl_{\mathcal C} \bbP^3$. We denote by $\alpha:X\to \bbP^3$ the projection, by $E$ the exceptional divisor and by $\widetilde{Q}$ the strict transform of $Q$. 

By \cite{LamyBlancWeak} the variety $X$ has a morphism to a (2,2) complete intersection $\overline X\subset \mathbb P^5$ contracting the $\widetilde Q$ to a line $L$:
\begin{align} \label{diag:1}
\xymatrix{
&E\subseteq X\supseteq \widetilde{Q}\ar[dr]^\beta\ar[dl]_\alpha \\
\mathcal C\subseteq\bbP^3 && \overline{X}\subseteq \bbP^5.
}
\end{align}

Set $H:=\alpha^*\mathcal O_{\mathbb P^3}$ be the pullback of the ample line bundle of $\mathbb P^3$. We have the following intersection numbers:
\begin{align*}
    E^3 = -22,\qquad  H^3 = 1,\qquad E\cdot H^2 = 0,\qquad E^2\cdot H = -5
\end{align*}

Note that $\alpha$ restricts to an isomorphism between $\widetilde Q$ and $Q$ while $\beta$ can be identified with a projection $\mathbb P^1\times \mathbb P^1\to \mathbb P^1$. Let $\ell_1$ be the rational equivalence class of a line that gets contracted and $\ell_2$ the one of a line that get mapped isomorphically to $L$ by $\beta$.
Let also $f$ denote the class of a fiber of $\alpha\vert_E\colon E\to\mathcal C$. We have the following intersection numbers:
\begin{align*}
    E\cdot \ell_1=3, \qquad E\cdot\ell_2=2,\qquad H\cdot f=0,\qquad \widetilde Q\cdot f=1 \\
\end{align*}

The following also holds:
\begin{align*}
    \widetilde Q \sim 2H-E,\qquad -K_X \sim 4H-E \sim 2\widetilde Q + E\qquad (-K_X)^3=26
\end{align*}

We will now estimate local delta invariants $\delta_Z(X)$ where $Z$ are irreducible subvarieties in $X$. Let $F$ be a prime divisor over $X$ and $Z=c_X(F)$ be the centre of the divisor on $X$.  If 
\begin{itemize}
    \item $Z$ is a surface, then $\delta_Z(X)>1$ since $X$ is \textit{divisorially K-stable}. See \cite[Theorem~10.1]{kentovolfuct}
    \item $Z$ is a curve, then we estimate $\delta_Z(X)$ in Subsection \ref{centre:curve}.
    \item $Z$ is a point not in $E \cap \tilde{Q}$, then we estimate $\delta_Z(X)$ in Subsection \ref{centre:point}.
\end{itemize}
This, in particular, proves Theorem \ref{thm:main} and the arguments used in proving Corollary \ref{cor:nofixedpts}. 

\subsection{Divisors with centre a curve}\label{centre:curve}\phantom.\\
Let $Z$ be an irreducible curve in $X$. If $Z\not \subset E$, then by Proposition \ref{prop:genpt}
 and Proposition \ref{delta: p in Q} there is a point $p\in Z$ away from $E$ for which $\delta_p(X)>1$. By definition of $\delta$-invariant we have $\delta_Z(X)\geq\delta_p(X)>1$. Hence we need only to consider curves $Z \subset E$. This is the content of Lemma~\ref{lem:centreCurve}, which is an estimate used in proving Corollary~\ref{cor:pointless}.

\begin{lemma}\label{lem:centreCurve}

    If $Z\subset E$ is a curve not meeting the curve $E\cap \widetilde Q$, then $\delta_Z(X)>1$.
\end{lemma}
\begin{proof}
    Let $Z$ be any curve in $E$ not meeting $s:=\widetilde Q\cap E$. Thus, its class in the Neron-Severi group of $E$ is given by $ns + 2nf$ where $f$ is the class of a fiber of the fibration $E\to \mathcal C$, and $n$ is a positive integer. We consider the flag given by $Z\subset E\subset X$ and apply \cite[Theorem 1.95 and Corollary~1.109]{ACCFKMGSSV}

    For this, we compute the negative and positive part of $-K_X-uE$ where $u\in \mathbb R_{\geq 0}$

    \begin{align*}
    P(u)=
        \begin{cases}
            2\widetilde Q + (1-u)E\quad \mbox{ if } u\in[0,\frac{1}{3}]\\
            (1-u)(3\widetilde Q + E)\quad \mbox{ if } u\in[\frac{1}{3}, 1],
        \end{cases}
    \end{align*}
    and $N(u)=(3u-1)\widetilde Q$ for $u\in[\frac{1}{3},1]$, and zero otherwise. 

    The restriction of $P(u)$ to $E$ is readily computed by noting that $E|_E=10f-s$. The positive and negative part of $P(u)|_E-vZ$ for $v\in \mathbb R_{\geq 0}$ is then computed to be:
    \begin{align*}
        P(u)=
        \begin{cases}
            (1+u-vn)s + (10-10u-2nv)f\quad \mbox{ if }  v\in[0, \frac{1+u}{n}]\mbox{ and } u\in[0,\frac{1}{3}]\\
            (2-2u-vn)s + (10-10u-2vn)f \quad \mbox{ if } v\in[\frac{2-2u}{n}] \mbox{ and } u\in[\frac{1}{3}, 1]
        \end{cases}
    \end{align*}
    A direct computation gives 
    \begin{align*}
        \frac{1}{S(V_{\bullet, \bullet}^E, Z)} = \frac{468n}{241},
    \end{align*}
    and since by \cite[Theorem~10.1]{kentovolfuct} we have $S_X(E)<1$, the claim follows.
\end{proof}

\subsection{Divisors with centre a point}\label{centre:point}
For different points $p$ on $X$, in this section we estimate $\delta_p(X)$.  
\begin{proposition}\label{delta: p in Q}
If $p$ is a point in $\widetilde Q$, then
\[
\delta_p(X) \geq \begin{cases}
    \frac{13}{12} & \text{if }p \in \tilde{Q}\backslash E,\\
    \frac{52}{49} & \text{if }p \in \tilde{Q}\cap E.
\end{cases}
\]
\end{proposition}

\begin{proof}
Let $p$ be a point in $\widetilde Q$. Let $L_1$ be the strict transform of the line in $Q$ through $p$ with class $\ell_1$ and similarly let $L_2$ be the strict transform of the line in $Q$ through $p$ with class $\ell_2$.
We consider the flag 
\begin{align*}    
p\in L_1 \subseteq \widetilde Q \subseteq X.
\end{align*}
For each $u \in \mathbb R_{\geq 0}$ we consider the divisor $-K_X-u\widetilde{Q}=E+(2-u)\widetilde{Q}$. Its pseudo-effective threshold is $\tau =2$ and its Zariski decomposition is
\begin{align*}
P(u)=
\begin{cases}
(4-2u)H+(u-1)E & \text{if}\ u\in [0,1],\\
(4-2u)H & \text{if}\ u\in [1,2],
\end{cases}
\quad \text{and} \quad 
N(u)=
\begin{cases}
0 & \text{if}\ u\in [0,1],\\
(u-1)E & \text{if}\ u\in [1,2].
\end{cases}
\end{align*}
We can then compute the volume to be 
\begin{align*}
\mathrm{vol}(-K_X-u\widetilde{Q})= \left(P(u)\right)^3=
\begin{cases}
-6u^2-12u+26& \text{if}\ u\in [0,1],\\
-8u^3+48u^2-96u+64 & \text{if}\ u\in [1,2].
\end{cases}
\end{align*}
Hence we get
\begin{align}\label{Flag1:3fold-to-surface}
    S_X(\widetilde Q) = \frac{1}{(-K_X)^3}\int_{0}^{\tau}\mathrm{vol}(-K_X-u\tilde{Q})du = \frac{10}{13}.
\end{align}
We now compute the value $S(V^{\widetilde Q}_{\bullet,\bullet};L_1)$. We consider for $v\in \mathbb R_{\geq 0}$ the linear system:

\begin{align*}
P(u)\vert_{\widetilde{Q}}-vL_1=
\begin{cases}
(2-v)L_1+(1+u)L_2 & \text{if}\ u\in [0,1]\\
(4-2u-v)L_1+(4-2u)L_2& \text{if}\ u\in [1,2].
\end{cases}
\end{align*}
The nefness and bigness of the above linear system is readily checked and its Zariski decomposition is given by:
\begin{align*}
P(u,v)=
\begin{cases}
(2-v)L_1+(1+u)L_2 & \text{if}\ u\in [0,1],\ v\in [0,2]\\
(4-2u-v)L_1+(4-2u)L_2& \text{if}\ u\in [1,2],\ v\in [0,4-2u],
\end{cases}N(u,v)=
\begin{cases}
0\\
0.
\end{cases}
\end{align*}
Hence
\begin{align*}
\mathrm{vol}(P(u)\vert_{\widetilde{Q}}-vL)=
\begin{cases}
2(2-v)(u+1) & \text{if}\ u\in [0,1],\ v\in [0,2]\\
4(4-2u-v)(2-u)& \text{if}\ u\in [1,2],\ v\in [0,4-2u].
\end{cases}
\end{align*}

We note that the restriction of the divisor $E$ to $\widetilde Q$ consists of an irreducible curve which is isomorphically mapped to $\mathscr C$ by the blow-up morphism $\alpha$. In particular, we see that $Z$ has no support on $L_1$ and the negative part $N(u)$ does not contribute in the formula \eqref{surface-to-curve} and we get:

\begin{align} \label{Flag1:surface-to-curve}
    S(V^{\widetilde Q}_{\bullet,\bullet};L_1) = \frac{12}{13}.
\end{align}
A direct computation gives the value of \eqref{curve-to-point}:
\begin{align}\label{Flag1:curve-to-point}
S(W_{\bullet,\bullet,\bullet}^{\widetilde Q,L_1};p) = \frac{10}{13}+F_p(W_{\bullet,\bullet,\bullet}^{\widetilde Q,L_1}).
\end{align}
Finally, we compute $F_p(W_{\bullet,\bullet,\bullet}^{\widetilde Q,L_1})$. Notice that $L_1$ is not contained in $Z$ so we have $N(u)=N'_{\widetilde Q}(u)$.  Let $m$ be the intersection multiplicity between $E\vert_{\widetilde{Q}}$ and $L_1$ at $p$. Then $m \in \{0,1,2,3\}$ and
 $\mathrm{ord}_p(N'_{\widetilde Q}(u)\vert_{L_1}) = m(u-1) \mbox{ if } u\in[1,2].$ For the value in \eqref{Fp} we therefore get:
\begin{align}\label{Flag1:Fp}
    F_p = \frac{m}{13}
\end{align}
and so
$$
S(W_{\bullet,\bullet,\bullet}^{\widetilde Q,L_1};p)=\frac{10+m}{13}.
$$
Hence,
$$
\delta_p(X)\geq \mathrm{min}\bigg\{\frac{13}{10},\ \frac{13}{12},\ \frac{13}{10+m}\bigg\}
$$
and so, with the flag above, we can only establish $\delta_p(X)>1$ whenever $m\not =3$, that is, when $L_1$ does not intersect $Z$ at an inflexion point. To resolve this issue we consider a $(1,3)$-weighted blowup of $\widetilde{Q}$ at $p$.  

From now on we assume $m=3$. Locally, we can assume that the curve $Z:=E\cap \widetilde Q$ is given by the equation $(z-w^3 = 0) \subset \mathbb A^2$ while $L_1$ is given by $(z=0) \subset \mathbb A^2$ and that $p$ is the point $(0,0)$. Let $\sigma \colon \widehat{Q} \rightarrow \widetilde{Q}$ be the $(1,3)$-blowup of $\widetilde{Q}$ at $p$ such that $\mathrm{wt}(z)=3$ and $\mathrm{wt}(w)=1$ and let $G$ be the $\sigma$-exceptional divisor. Denote by $\widehat Z$ its strict transform in $\widehat Q$. Then $\widehat{Z}$ intersects the exceptional curve $G$ in one regular point and the following hold:
\begin{align*}
&\widehat{Z}=\sigma^*(Z)-3G, \quad K_{\widehat{Q}}=\sigma^*(K_{\widetilde Q})+3G,\\
&\widehat{L_1}=\sigma^*L_1-3G, \quad\widehat{L_2}=\sigma^*L_2-G.
\end{align*}
We have the following intersection numbers,

\begin{center}
\begin{tabular}{ c|ccc }  

 & $G$ & $\widehat{L_1}$ & $\widehat{L_2}$ \\ 
 \hline
 $G$ & $-\frac{1}{3}$ & $1$ & $\frac{1}{3}$ \\ 
 $\widehat{L_1}$ &  & $-3$ & $0$ \\ 
 $ \widehat{L_2}$ &  &  & $-\frac{1}{3}$ \\ 
 
\end{tabular}
\end{center}

Notice that $G$ has one singular point $P$ of type $\frac{1}{3}(1,1)$. In particular, the different $\Delta_G$ defined by:
\begin{align*}
(K_{\widehat Q} + G)\vert_G = K_{G} + \Delta_G\quad \mbox{ is given by }\quad \Delta_G = \frac{2}{3}P.
\end{align*}

We start by computing $S(V_{\bullet,\bullet}^{\widehat{Q}},G)$. We consider the linear system 
\begin{align*}
    \sigma^*(P(u)\vert_{\widetilde Q})-vG=
    \begin{cases}
    2\widehat{L_1}+(u+1)\widehat{L_2}+(7+u-v)G &\text{if } u\ \in[0,1],\\
    (4-2u)(\widehat{L_1}+\widehat{L_2})+(16-8u-v)G&\text{if } u\ \in[1,2].
    \end{cases}
\end{align*}
Using the intersection numbers above we can compute the Zariski decomposition of  $\sigma^*(P(u)\vert_{\widetilde Q})-vG$ as 

\begin{align*}
\widetilde P(u,v)=
 \begin{cases}
    2\widehat{L_1}+(u+1)\widehat{L_2}+(7+u-v)G &\text{if } u\ \in[0,1],\, v \in [0,1+u]\\
    (7+u-v)\big(\frac{\widehat{L_1}}{3}+G\big)+(u+1)\widehat{L_2}&\text{if } u\ \in[0,1],\, v \in [1+u,6]\\
     (7+u-v)\big(\frac{\widehat{L_1}}{3}+\widehat{L_2}+G\big)&\text{if } u\ \in[0,1],\, v \in [6,7+u]\\
     (4-2u)(\widehat{L_1}+\widehat{L_2})+(16-8u-v)G&\text{if } u\ \in[1,2],\, v \in [0,4-2u]\\
      (16-8u-v)\big(\frac{\widehat{L_1}}{3}+G\big)+(4-2u)\widehat{L_2}&\text{if } u\ \in[1,2],\, v \in [4-2u, 12-6u]\\
      (16-8u-v)(\widehat{L_1}+3\widehat{L_2}+3G)&\text{if } u\ \in[1,2],\, v \in [12-6u,16-8u].
    \end{cases}
\end{align*}
and
\begin{align*}
\widetilde N(u,v)=
 \begin{cases}
   0 &\text{if } u\ \in[0,1],\, v \in [0,1+u]\\
    \frac{1}{3}(v-u-1)\widehat{L_1}&\text{if } u\ \in[0,1],\, v \in [1+u,6]\\
     \frac{1}{3}(v-u-1)\widehat{L_1}+(v-6)\widehat{L_2}&\text{if } u\ \in[0,1],\, v \in [6,7+u]\\
    0&\text{if } u\ \in[1,2],\, v \in [0,4-2u]\\
      \frac{1}{3}(v+2u-4)\widehat{L_1}&\text{if } u\ \in[1,2],\, v \in [4-2u, 12-6u]\\
      \frac{1}{3}(v+2u-4)\widehat{L_1}+(6u+v-12)\widehat{L_2}&\text{if } u\ \in[1,2],\, v \in [12-6u,16-8u].
    \end{cases}
\end{align*}

We note that for $u\in [1,2]$ the contribution of the negative part in \eqref{Rmk1.7.32-surface-to-curve} is $\mathrm{ord}_G(\sigma^* N(u)\vert_{\widetilde{Q}}) = \mathrm{ord}_G((u-1)(\widehat{Z}+3G)) = 3(u-1)$. We have $A_{\widetilde{Q}}(G) = 4$ so 
\begin{align}\label{eq:QSurf-to-curve}
    \frac{A_{\widetilde{Q}}(G)}{S(V^{\widehat{Q}}_{\bullet,\bullet};G)} = \frac{52}{49}.
\end{align}
We now compute $S(W_{\bullet,\bullet,\bullet}^{\widetilde Q,G};q)$ for every point $q \in G$. Notice that $\mathrm{ord}_q\Delta_G=0$ if $q$ is away from the singular point of $G$ and $\mathrm{ord}_q\Delta_G=\frac{2}{3}$ if $q$ is the singular point of $G$. Hence, from equation \eqref{Rmk1.7.32-curve-to-point} we have
$$
        \frac{1-\mathrm{ord}_q\Delta_{G}}{S(W_{\bullet,\bullet,\bullet}^{\widetilde Q,G};q)} = \begin{cases}
            \frac{936}{217+936 F_q(W_{\bullet,\bullet,\bullet}^{\widetilde Q, G})}  &\text{ if } q \not = P \\
             \frac{312}{217+936 F_q(W_{\bullet,\bullet,\bullet}^{\widetilde Q, G})}       &\text{ if } q = P.
        \end{cases}   .
 $$
We split the computations of $F_q(W_{\bullet,\bullet,\bullet}^{\widetilde Q, G})$ into the following four cases depending on the position of $q$ in $G$:
\begin{enumerate}
    \item $q \not \in \widehat{Z} \cup \widehat{L_1} \cup \widehat{L_2}$. In this case $F_q(W_{\bullet,\bullet,\bullet}^{\widetilde Q, G})=0$ and so
    $$
        \frac{1-\mathrm{ord}_q\Delta_{G}}{S(W_{\bullet,\bullet,\bullet}^{\widetilde Q,G};q)} = \frac{936}{217+936 F_q(W_{\bullet,\bullet,\bullet}^{\widetilde Q, G})}  = \frac{936}{217}.
 $$
    \item $q = \widehat{Z} \cap G$. In this case, $\mathrm{ord}_q(\widetilde{N}(u,v)\vert_{G}) = 0$ and
     \begin{align*}
    \mathrm{ord}_q(N_{\widehat{Q}}'(u)\vert_{G}) =\begin{cases}
            0  &\text{ if } u\in[0,1] \\
            1       &\text{ if } u\in[1,2].
        \end{cases}
    \end{align*}
    Then we can readily compute 
    $$
    F_q(W_{\bullet,\bullet,\bullet}^{\widetilde Q, G}) = \frac{1}{13}
    $$
    and so 
     $$
        \frac{1-\mathrm{ord}_q\Delta_{G}}{S(W_{\bullet,\bullet,\bullet}^{\widetilde Q,G};q)} = \frac{936}{217+936 F_q(W_{\bullet,\bullet,\bullet}^{\widetilde Q, G})}  = \frac{936}{289}.
 $$
    \item $q = \widehat{L_1} \cap G$. In this case $\mathrm{ord}_q(N_{\widehat{Q}}'(u)\vert_{G})=0$ and 
 \begin{align*}
    \mathrm{ord}_q(\widetilde{N}(u,v)\vert_{G}) = \begin{cases}
   0 &\text{if } u\ \in[0,1],\, v \in [0,1+u]\\
     \frac{1}{3}(v-u-1)&\text{if } u\ \in[0,1],\, v \in [1+u,7+u]\\
    0&\text{if } u\ \in[1,2],\, v \in [0,4-2u]\\
      \frac{1}{3}(v+2u-4)&\text{if } u\ \in[1,2],\, v \in [4-2u,16-8u].
    \end{cases}
    \end{align*}
     Then we can readily compute 
    $$
    F_q(W_{\bullet,\bullet,\bullet}^{\widetilde Q, G}) = \frac{647}{936}
    $$
    and so 
     $$
        \frac{1-\mathrm{ord}_q\Delta_{G}}{S(W_{\bullet,\bullet,\bullet}^{\widetilde Q,G};q)} =  \frac{936}{217+936 F_q(W_{\bullet,\bullet,\bullet}^{\widetilde Q, G})}  = \frac{13}{12}.
 $$

    \item $q = \widehat{L_2} \cap G$. This is the singular point $P$ of $G$. In this case $\mathrm{ord}_q(N_{\widehat{Q}}'(u)\vert_{G})=0$ and 
 \begin{align*}
    \mathrm{ord}_q(\widetilde{N}(u,v)\vert_{G}) = \begin{cases}
   0 &\text{if } u\ \in[0,1],\, v \in [0,6]\\
     \frac{1}{3}(v-6)&\text{if } u\ \in[0,1],\, v \in [6,7+u]\\
    0&\text{if } u\ \in[1,2],\, v \in [0,12-6u]\\
      \frac{1}{3}(6u+v-12)&\text{if } u\ \in[1,2],\, v \in [12-6u,16-8u].
    \end{cases}
    \end{align*}
     Then we can readily compute 
    $$
    F_q(W_{\bullet,\bullet,\bullet}^{\widetilde Q, G}) = \frac{23}{936}
    $$
    and so 
     $$
        \frac{1-\mathrm{ord}_q\Delta_{G}}{S(W_{\bullet,\bullet,\bullet}^{\widetilde Q,G};q)} =  \frac{312}{217+936 F_q(W_{\bullet,\bullet,\bullet}^{\widetilde Q, G})}  = \frac{13}{10}.
 $$
\end{enumerate}
Therefore,
 \begin{align} \label{eq:QCurv-to-pt}
\mathrm{min}_{q\in G}\frac{1-\mathrm{ord}_q\Delta_{G}}{S(W_{\bullet,\bullet,\bullet}^{\widetilde Q,G};q)}=\frac{13}{12}.
\end{align}

From Theorem \ref{blow-up-of-surface-formula} and combining the computations from \eqref{Flag1:3fold-to-surface}, \eqref{eq:QSurf-to-curve} and \eqref{eq:QCurv-to-pt}, it follows that 
$$
\delta_p(\widetilde{Q}) \geq \mathrm{min}\bigg\{\frac{13}{12},\frac{52}{49},\frac{13}{10} \bigg\} = \frac{52}{49}.
$$    
\end{proof}

In light of the following remark we need to have special care in the choice of the flag to compute an estimate for $\delta_p$ when $p$ is a general point not lying in $E\cup\widetilde Q$.


\begin{remark}\label{rmk:Ciro}
    There are smooth curves $\mathcal C\subset \mathbb P^3$ of genus 2 and degree 5 for which a point in $\mathbb P^3$ lies on no 2-secants to the curve $\mathcal C$ (it necessarily lies on a tangent though). For this, it is better to to think of a fixed smooth quadric $Q\subset \mathbb P^3$ and let the genus 2 and degree 5 curve $\mathcal C$ vary on it.
    The projection $pr_p\colon \mathbb P^3\dashrightarrow \mathbb P^2$ from a point $p\in \mathbb P^3\setminus Q$ maps $\mathcal C$ to a plane degree 5 curve. By comparing the genusses of $\mathcal C$ and a smooth planar degree 5 curve, one sees that the image of $\mathcal C$ has at most 4 nodes. It is then enough to show the existence of a curve $\mathcal C$ for which $p$ lies on 4 tangents. For this, consider the polar plane to $p$ intersecting $Q$ in the conic which is the ramification locus of the restriction $pr_Q\colon Q\to \mathbb P^2$. Let $q_1, q_2,q_3, q_4$ be four points on that conic. Imposing that $\mathcal C$ passes through those points with tangent direction the one determined by the line $pq_i$ amounts to give 8 conditions on the 11-dimensional family of $(2,3)$-curves on $Q\simeq \mathbb P^1\times \mathbb P^1$.
\end{remark}

\begin{proposition} \label{prop:genpt}
 If $p \in X \backslash (E \cup \tilde{Q})$, then 
\[
\delta_p(X) \geq \frac{208}{205}.
 \]
\end{proposition}

\begin{proof}
Let $S:=H_{\mathbb{P}^3}$ in $\mathbb{P}^3$ be the general hyperplane section of $\mathbb{P}^3$ such that $\alpha(p) \in S$. Let the strict transform of $S$ on $X$ be denoted by $\tilde{S}$. Note that $p\in \tilde{S}$ by how the surface $S$ is chosen. 

The Zariski decomposition of the linear system $-K_X-u\widetilde S$ is given by:\bigskip

\[
P(u)=
\begin{cases}
(4-u)H-E \equiv (2-u/2)\widetilde Q + (1-u/2)E &\text{if}\ u\in [0,1],\\
(2-u)(3H-E) \equiv (1-u/2)(3\widetilde Q + E) &\text{if}\ u\in [1,2],
\end{cases}
\]
and
\[
N(u)=
\begin{cases}
0 & \text{if}\ u\in [0,1],\\
(u-1)\widetilde Q & \text{if}\ u\in [1,2].
\end{cases}
\]
This then gives the following:
\begin{align*}
    S_X(\widetilde S)= \frac{1}{(-K_X)^3}\int_{0}^{\tau(\widetilde S)}\mathrm{vol}(-K_X-u\widetilde S)du = \frac{57}{104}.
    \end{align*}

\textbf{Case 1: When secant line to $\mathcal{C}$ through $\alpha(p)$ exists.}\label{general point-2 secant exists}\phantom.

Let $p\in X \backslash  E\cup \widetilde Q$ be a point such that $\alpha(p)$ lies on a 2-secant $Z$ of the curve $\mathcal C$. Let the surface $S$ as above, be taken to be containing $Z$. We then consider the flag \[p\in \tilde{Z} \subset \widetilde S \in X\] where the strict transform $\widetilde Z$ of $Z$ in $X$ is a $(-1)$-curve.

\textit{Notations:} Note that $\tilde{S}$ is the blow up of $S$ in 5 distinct points. Let us denote by $p_1,p_2,\cdots,p_5\in \Pi$ the blown-up points on $S$, by $e_1,e_2,\cdots,e_5$ the associated exceptional curves in Pic$(\widetilde S)$ and by $L_{ij}$ for $1\leq i <j \leq 5$, the strict transform of lines on $S$ that pass through points $p_i,\ p_j$. Since $Z$ is 2-secant to the curve, let $Z=L_{12}$ in $S$ be the line that passes through points $p_1,\ p_2$. Let $\widetilde Z=L_{12}$ be the strict transform of the line $L_{12} \in S$  and has class $h-e_1-e_2 \in\ $Pic$(\widetilde S)$, where $h$ is given by pulling back $\mathcal O_{\mathbb P^3}(1)$.

We consider the linear system $P(u)|_{\widetilde S}-v\widetilde Z$ for $v\in \mathbb R_{>0}$. Its Zariski decomposition for $u\in [0,1]$ is given by:\bigskip

$
P(u,v)=
\begin{cases}
(4-u-v)h-(1-v)(e_1+e_2)-\sum_3^5 e_i   &\text{if}\ v\in [0,1],\\
(4-u-v)h-e_3-e_4-e_5                  &\text{if}\ v\in [1,2-u],\\
(5-2v-2u)(2h-e_3-e_4-e_5)              &\text{if}\ v\in [2-u,5/2-u],
\end{cases}
$
and \bigskip

$N(u,v)=
\begin{cases}
0                  &\text{if}\ v\in [0,1],\\
(v-1)(e_1+e_2)     &\text{if}\ v\in [1,2-u],\\
(v-1)(e_1+e_2) + (u+v-2)(L_{34}+L_{35}+L_{45})    &\text{if}\ v\in [2-u,5/2-u],
\end{cases}
$\bigskip

and for $u\in [1,2]$ is given by:\bigskip

$
P(u,v)=
\begin{cases}
(2-u)(3h-\sum_1^5 e_i) -v(h-e_1-e_2)   &\text{if}\ v\in [0,2-u],\\
(6-3u-2v)(2h-e_3-e_4-e_5)                  &\text{if}\ v\in [2-u,3-3u/2],
\end{cases}
$and
\bigskip

$N(u,v)=
\begin{cases}
0                  &\text{if}\ u\in [0,2-u],\\
(v+u-2)(e_1+e_2+L_{34}+L_{35}+L_{45})     &\text{if}\ u\in [2-u,3-3u/2].
\end{cases}
$\bigskip

Since $\tilde{Z} \not \subset \mathrm{Supp}(\tilde{Q}\vert_{\widetilde S})$, $\mathrm{ord}_{\tilde{Z}}(N(u) \vert_{\widetilde S})$=0. Therefore, 
\[
 S(V^{\widetilde{S}}_{\bullet,\bullet};\tilde{Z}) =\frac{183}{208}.
\]

Finally, we have $ S(W^{\widetilde{S}, \tilde{Z}}_{\bullet, \bullet}; p)$. Note that since $p \not\in \tilde{Q}$, $\mathrm{ord}_p(N'_{\widetilde{S}}(u)\vert_{\tilde{Z}})=0$. We then have the following cases:
\begin{itemize}
    \item \textbf{$p \in \tilde{Z} \backslash \bigcup_{3\leq i<j\leq 5} L_{ij}$}: Here, $F_p(W_{\bullet,\bullet,\bullet}^{\widetilde{S}, \tilde{Z}})=0$ and  $ S(W^{\widetilde{S}, \tilde{Z}}_{\bullet, \bullet}; p)=\frac{25}{26}.$
    \item \textbf{$p \in \tilde{Z} \cap L_{ij}\ \text{for}\ {3\leq i<j\leq 5}$}: $F_p(W_{\bullet,\bullet,\bullet}^{\widetilde{S}, \tilde{Z}})=\frac{5}{208}$ and $S(W^{\widetilde{S}, \tilde{Z}}_{\bullet, \bullet}; p)= \frac{205}{208}$.
\end{itemize}
Therefore, 
\[
\delta_p(X) \geq \mathrm{min}\Big\{\frac{104}{57}, \frac{208}{183},\frac{208}{205}\Big\}=\frac{208}{205}>1.
\]

\textbf{Case 2: When no secant line to $\mathcal{C}$ through $\alpha(p)$ exists.}\phantom.

Let $p \in X \backslash (E \cup \tilde{Q})$ be such that $\alpha(p)$ does not lie on a 2-secant to the curve $\mathcal{C}$. Let $S:=H_{\mathbb{P}^3}$ be a hyperplane containing the point $\alpha(p)$ and tangent to the curve $\mathcal{C}$ at the point $p_1$ where $S \cap \mathcal{C}=\{p_1,p_2,p_3,p_4\}$. That is, $\mathrm{mult}_{p_1}(S.\mathcal{C})=2$ and $S$ contains the tangent line $L:=T_{p_1}\mathcal{C}$ such that $\alpha(p) \in L$.  

\textit{Notations:} Let $\widetilde{L},\ \widetilde{S}$ be the strict transforms of the line $L$ and the surface $S$ on $X$.Note that $\tilde{S}$ is the normal blow up of the points $p_i\ \text{for}\ 2 \leq i \leq 4$ and a $(1,2)$ blow up at the point $p_1$ and  since $L\in S$ is a line tangent to the curve $\mathcal{C}$ at $p_1$, the class of $\tilde{L}:=h-2e_1$, where $h$ is given by pulling back $\mathcal O_{\mathbb P^3}(1)$. 

We then consider the flag 
\[p \in \widetilde{L} \subset \tilde{S} \subset X\]  


Consider the linear system $P(u)|_{\widetilde S}-v\widetilde L$ for $v\in \mathbb R_{>0}$. Its Zariski decomposition for $u\in [0,1]$ is given by:\bigskip

$
P(u,v)=
\begin{cases}
(4-u-v)h-2(1-v)e_1-\sum_2^4 e_i   &\text{if}\ v\in [0,1],\\
(4-u-v)h-e_2-e_3-e_4                 &\text{if}\ v\in [1,2-u],\\
(5-2v-2u)(2h-e_2-e_3-e_4)              &\text{if}\ v\in [2-u,5/2-u],
\end{cases}
$
and \bigskip

$N(u,v)=
\begin{cases}
0                  &\text{if}\ v\in [0,1],\\
2(v-1)e_1    &\text{if}\ v\in [1,2-u],\\
2(v-1)e_1 + (u+v-2)(L_{23}+L_{24}+L_{34})    &\text{if}\ v\in [2-u,5/2-u],
\end{cases}
$\bigskip

and for $u\in [1,2]$ is given by:\bigskip

$
P(u,v)=
\begin{cases}
(6-3u-v)h+(2u+2v-4)e_1-(2-u)(\sum_2^4 e_i)  &\text{if}\ v\in [0,2-u],\\
(6-3u-2v)(2h-e_3-e_4-e_5)                  &\text{if}\ v\in [2-u,3-3u/2],
\end{cases}
$and
\bigskip

$N(u,v)=
\begin{cases}
0                  &\text{if}\ u\in [0,2-u],\\
(2u+2v-4)e_1+(v+u-2)(L_{23}+L_{24}+L_{34})     &\text{if}\ u\in [2-u,3-3u/2].
\end{cases}
$\bigskip

Since $\widetilde{L} \not \subset \mathrm{Supp}(\tilde{Q}\vert_{\widetilde S})$, $\mathrm{ord}_{\widetilde{L}}(N(u) \vert_{\widetilde S})$=0. Therefore, 
\[
 S(V^{\widetilde{S}}_{\bullet,\bullet};\widetilde{L}) =\frac{183}{208}.
\]

Finally, we have $S(W_{\bullet,\bullet}^{\tilde{S},\widetilde{L}},p)$. 

Note that since $p \not\in \tilde{Q}$, $\mathrm{ord}_p(N'_{\widetilde{S}}(u)\vert_{\tilde{L}})=0$. We then have the following cases:
\begin{itemize}
    \item \textbf{$p \in \tilde{L} \backslash \bigcup_{2\leq i<j\leq 4} L_{ij}$}: Here, $F_p(W_{\bullet,\bullet,\bullet}^{\widetilde{S}, \tilde{L}})=0$ and  $ S(W^{\widetilde{S}, \tilde{L}}_{\bullet, \bullet}; p)=\frac{25}{26}.$
    \item \textbf{$p \in \tilde{L} \cap L_{ij}\ \text{for}\ {2\leq i<j\leq 4}$}: $F_p(W_{\bullet,\bullet,\bullet}^{\widetilde{S}, \tilde{L}})=\frac{5}{208}$ and $S(W^{\widetilde{S}, \tilde{L}}_{\bullet, \bullet}; p)= \frac{205}{208}$.
\end{itemize}
Therefore, 
\[
\delta_p(X) \geq \mathrm{min}\Big\{\frac{104}{57}, \frac{208}{183},\frac{208}{205}\Big\}=\frac{208}{205}>1.
\]
\end{proof}

\begin{proof}(of Theorem~\ref{thm:main})
    Let $p$ be a point not in $E\setminus \widetilde Q$, then if $p\in \widetilde Q$, one has $\delta_p(X)>\frac{52}{49}$ by Proposition~\ref{delta: p in Q}. While if $p\not\in E\cup \widetilde Q$, then $\delta_p(X)>\frac{208}{205}$ by Proposition~\ref{prop:genpt}. This shows the claim.
\end{proof}

\bibliography{literatur}

\newcommand{\etalchar}[1]{$^{#1}$}
\begin{thebibliography}{ACKM24}

\bibitem[ACC{\etalchar{+}}21]{ACCFKMGSSV}
Carolina Araujo, Ana-Maria Castravet, Ivan Cheltsov, Kento Fujita, Anne-Sophie
  Kaloghiros, Jesus Martinez-Garcia, Constantin Shramov, Hendrik S{\"u}{\ss},
  and Nivedita Viswanathan.
\newblock The {C}alabi problem for fano threefolds.
\newblock available at
  \href{https://archive.mpim-bonn.mpg.de/id/eprint/4589/1/mpim-preprint_2021-31.pdf}{https://archive.mpim-bonn.mpg.de/id/eprint/4589/1/mpim-preprint\_2021-31.pdf},
  2021.

\bibitem[ACKM24]{abban2024kstabilitypointlessdelpezzo}
Hamid Abban, Ivan Cheltsov, Takashi Kishimoto, and Frederic Mangolte.
\newblock K-stability of pointless del pezzo surfaces and fano 3-folds, 2024.

\bibitem[AZ22]{AZ21}
Hamid Abban and Ziquan Zhuang.
\newblock K-stability of {F}ano varieties via admissible flags.
\newblock {\em Forum Math. Pi}, 10:Paper No. e15, 43, 2022.

\bibitem[BJ20]{blumjonsson}
Harold Blum and Mattias Jonsson.
\newblock Thresholds, valuations, and {K}-stability.
\newblock {\em Adv. Math.}, 365:107062, 57, 2020.

\bibitem[BL12]{LamyBlancWeak}
J\'{e}r\'{e}my Blanc and St\'{e}phane Lamy.
\newblock Weak {F}ano threefolds obtained by blowing-up a space curve and
  construction of {S}arkisov links.
\newblock {\em Proc. Lond. Math. Soc. (3)}, 105(5):1047--1075, 2012.

\bibitem[BL22]{belousov2022kstability}
Grigory Belousov and Konstantin Loginov.
\newblock K-stability of fano threefolds of rank 4 and degree 24.
\newblock {\em ar{X}iv:2206.12208}, 2022.

\bibitem[BX19]{blum2019uniqueness}
Harold Blum and Chenyang Xu.
\newblock Uniqueness of {K}-polystable degenerations of fano varieties.
\newblock {\em Annals of Mathematics}, 190(2):609--656, 2019.

\bibitem[CDF22]{cheltsov2022kstable}
Ivan Cheltsov, Elena Denisova, and Kento Fujita.
\newblock K-stable smooth fano threefolds of {P}icard rank two.
\newblock {\em ar{X}iv:2210.14770}, 2022.

\bibitem[CDS15a]{CDS1}
Xiuxiong Chen, Simon Donaldson, and Song Sun.
\newblock K\"ahler-{E}instein metrics on {F}ano manifolds. {I}: {A}pproximation
  of metrics with cone singularities.
\newblock {\em J. Amer. Math. Soc.}, 28(1):183--197, 2015.

\bibitem[CDS15b]{CDS2}
Xiuxiong Chen, Simon Donaldson, and Song Sun.
\newblock K\"ahler-{E}instein metrics on {F}ano manifolds. {II}: {L}imits with
  cone angle less than {$2\pi$}.
\newblock {\em J. Amer. Math. Soc.}, 28(1):199--234, 2015.

\bibitem[CDS15c]{CDS3}
Xiuxiong Chen, Simon Donaldson, and Song Sun.
\newblock K\"ahler-{E}instein metrics on {F}ano manifolds. {III}: {L}imits as
  cone angle approaches {$2\pi$} and completion of the main proof.
\newblock {\em J. Amer. Math. Soc.}, 28(1):235--278, 2015.

\bibitem[CFKO23]{cheltsov2023kstable}
Ivan Cheltsov, Kento Fujita, Takashi Kishimoto, and Takuzo Okada.
\newblock K-stable divisors in
  $\mathbb{P}^1\times\mathbb{P}^1\times\mathbb{P}^2$ of degree $(1,1,2)$.
\newblock {\em ar{X}iv:2206.08539}, 2023.

\bibitem[CFKP23]{CheltsovFujitaKishimotoPark}
Ivan Cheltsov, Kento Fujita, Takashi Kishimoto, and Jihun Park.
\newblock K-stable {F}ano 3-folds in the {F}amilies {N}o.2.18 and {N}o.3.4.
\newblock {\em to appear soon on ar{X}iv.}, 2023.

\bibitem[CP22]{CheltsovPark2022}
Ivan Cheltsov and Jihun Park.
\newblock K-stable {F}ano threefolds of rank 2 and degree 30.
\newblock {\em Eur. J. Math.}, 8(3):834--852, 2022.

\bibitem[CPS19]{3foldsInfAut}
Ivan Cheltsov, Victor Przyjalkowski, and Constantin Shramov.
\newblock Fano threefolds with infinite automorphism groups.
\newblock {\em Izvestiya: Mathematics}, 83, 08 2019.

\bibitem[Den22]{denisova2022kstability}
Elena Denisova.
\newblock On {K}-stability of $\mathbb{P}^3$ blown up along the disjoint union
  of a twisted cubic curve and a line.
\newblock {\em ar{X}iv:2202.04421}, 2022.

\bibitem[FO18]{kentoodaka}
Kento Fujita and Yuji Odaka.
\newblock On the {K}-stability of {F}ano varieties and anticanonical divisors.
\newblock {\em Tohoku Math. J. (2)}, 70(4):511--521, 2018.

\bibitem[Fuj16]{kentovolfuct}
Kento Fujita.
\newblock On {K}-stability and the volume functions of {$\mathbb{Q}$}-{F}ano
  varieties.
\newblock {\em Proc. Lond. Math. Soc. (3)}, 113(5):541--582, 2016.

\bibitem[Fuj19]{kentoplt}
Kento Fujita.
\newblock Uniform {K}-stability and plt blowups of log {F}ano pairs.
\newblock {\em Kyoto J. Math.}, 59(2):399--418, 2019.

\bibitem[Fuj22]{Kentorank3deg28}
Kento Fujita.
\newblock {On {K}-Stability for Fano Threefolds of Rank 3 and Degree 28}.
\newblock {\em International Mathematics Research Notices}, 07 2022.
\newblock rnac190.

\bibitem[GGV24]{215}
Tiago~Duarte Guerreiro, Luca Giovenzana, and Nivedita Viswanathan.
\newblock On {K}-stability of blown up along a (2,3) complete intersection.
\newblock {\em Journal of the London Mathematical Society}, 110(1):e12961,
  2024.

\bibitem[Jun]{Zhao24}
Zhao Junyan.
\newblock {K}-stability of {T}haddeus’ moduli of stable bundle pairs on genus
  two curves.
\newblock in preparation.

\bibitem[Liu23]{Liu-Rank2Deg14}
Yuchen Liu.
\newblock K-stability of {F}ano threefolds of rank 2 and degree 14 as double
  covers.
\newblock {\em Math. Z.}, 303(2):Paper No. 38, 9, 2023.

\bibitem[LX19]{LiuXu19}
Yuchen Liu and Chenyang Xu.
\newblock K-stability of cubic threefolds.
\newblock {\em Duke Math. J.}, 168(11):2029--2073, 2019.

\bibitem[Mal23]{Malbon}
Joseph Malbon.
\newblock {K}-stable {F}ano threefolds of rank 2 and degree 28.
\newblock {\em ar{X}iv:2405.08165}, 2023.

\bibitem[Tia97]{tian97}
Gang Tian.
\newblock K\"{a}hler-{E}instein metrics with positive scalar curvature.
\newblock {\em Invent. Math.}, 130(1):1--37, 1997.

\bibitem[Tia15]{Tian}
Gang Tian.
\newblock K-stability and {K}\"{a}hler-{E}instein metrics.
\newblock {\em Comm. Pure Appl. Math.}, 68(7):1085--1156, 2015.

\bibitem[Zhu21]{equivkstab}
Ziquan Zhuang.
\newblock Optimal destabilizing centers and equivariant {K}-stability.
\newblock {\em Inventiones mathematicae}, 226(1):195--223, 2021.

\end{thebibliography}
\bibliographystyle{alpha}

\end{document}